\documentclass{amsart}
\usepackage{amssymb, amsmath, mathrsfs, verbatim, url, bbm}

\theoremstyle{plain}
\newtheorem{theorem}{Theorem}[section]
\newtheorem{lemma}[theorem]{Lemma}

\newtheorem{corollary}[theorem]{Corollary}

\newcommand{\MainTheoremName}{Main Theorem}

\theoremstyle{definition}
\newtheorem{definition}[theorem]{Definition}
\newtheorem{example}[theorem]{Example}

\theoremstyle{remark}
\newtheorem*{remark}{Remark}

\newtheorem{question}[theorem]{Question}

\numberwithin{equation}{section}

\DeclareMathOperator{\upset}{\uparrow\!}
\DeclareMathOperator{\downset}{\downarrow\!}

\DeclareMathOperator{\ran}{ran}
\DeclareMathOperator{\cf}{cf}

\DeclareMathOperator{\Clop}{Clop}

\DeclareMathOperator{\Hull}{Hull}

\newcommand{\nbd}{\nobreakdash}

\newcommand{\powset}[1]{\mathcal{P}(#1)}
\newcommand{\card}[1]{\left\lvert #1\right\rvert}

\newcommand{\la}{\langle}
\newcommand{\ra}{\rangle}

\newcommand{\inv}[1]{#1^{-1}}
\newcommand{\closure}[1]{\overline{#1}}
\newcommand{\restrict}{\upharpoonright}

\newcommand{\weight}[1]{w(#1)}

\newcommand{\character}[1]{\chi(#1)}

\newcommand{\wma}{we may assume}
\newcommand{\Wma}{We may assume}

\newcommand{\elemsub}{\prec}

\newcommand{\supp}[1]{\mathrm{supp}(#1)}

\newcommand{\mcA}{\mathcal{ A}}
\newcommand{\mcB}{\mathcal{ B}}

\newcommand{\mcD}{\mathcal{ D}}
\newcommand{\mcE}{\mathcal{ E}}
\newcommand{\mcF}{\mathcal{ F}}
\newcommand{\mcG}{\mathcal{ G}}

\newcommand{\mcT}{\mathcal{ T}}
\newcommand{\mcU}{\mathcal{ U}}

\newcommand{\om}{\omega}

\newcommand{\al}{\aleph}
\newcommand{\ka}{\kappa}
\newcommand{\lm}{\lambda}
\newcommand{\ie}{{\it i.e.},}
\newcommand{\eg}{{\it e.g.},}

\newcommand{\locsplit}[3][.]{\mathrm{split}_{#2}^{\ifx#1.{}\else{#1}\fi}(#3)}

\newcommand{\alo}{\al_0}
\newcommand{\bigh}{H(\theta)}
\newcommand{\ult}[1]{\mathrm{Ult}\left(#1\right)}
\newcommand{\scepin}{\v S\v cepin}
\newcommand{\lkfn}{$(\lm,\ka)$\nbd-FN}
\newcommand{\lkog}{$(\lm,\ka)$\nbd-openly generated}
\newcommand{\laog}{$(\lm,\alo)$\nbd-openly generated}
\newcommand{\lka}{$(\lm,\ka)$\nbd-adic}
\newcommand{\lkb}{$(\lm,\ka)$\nbd-base}

\newcommand{\trileq}{\trianglelefteq}

\title{The $(\lm,\ka)$-Freese-Nation Property for boolean algebras and compacta}

\author{
David Milovich}
\address{
Department of Engineering, Mathematics and Physics \\
Texas A\&M International University \\
5201 University Blvd\\
Laredo, TX\\
78041 USA}

\email{david.milovich@tamiu.edu}

\subjclass[2000]{Primary: 03E04; 06E05; Secondary: 54B10, 54A25, 54C10}

\keywords{Freese-Nation Property, $\ka$-FN, openly generated, base properties, compact, dyadic, hyperspace, boolean algebra, $\pi$-character, caliber}

\begin{document}

\begin{abstract}
We study a two-parameter generalization 
of the Freese-Nation Property of boolean algebras 
and its order-theoretic and topological consequences.

For every regular infinite $\ka$, 
the $(\ka,\ka)$-FN, the $(\ka^+,\ka)$-FN, and the $\ka$-FN 
are known to be equivalent;
we show that the family of properties $(\lm,\mu)$-FN 
for $\lm>\mu$ form a true two-dimensional hierarchy
that is robust with respect to coproducts, retracts,
and the exponential operation.

The $(\ka,\alo)$-FN in particular 
has strong consequences for base properties of compacta
(stronger still for homogeneous compacta), 
and these consequences have natural duals 
in terms of special subsets of boolean algebras.
We show that the $(\ka,\alo)$-FN also leads
to a generalization of the equality of weight
and $\pi$-character in dyadic compacta.

Elementary subalgebras and their duals, elementary quotient spaces, 
were originally used to define the \lkfn\  
and its topological dual, which naturally generalized from Stone spaces 
to all compacta, thereby generalizing \scepin's notion 
of openly generated compacta.  We introduce a simple combinatorial 
definition of the \lkfn\  that is equivalent to the original
for regular infinite cardinals $\lm>\ka$.
\end{abstract}

\maketitle

\section{Introduction}

Fuchino, Koppelberg, and Shelah~\cite{fks}\ introduced 
the $\ka$-Freese-Nation property, of $\ka$-FN, 
as a natural generalization of the classical Freese-Nation 
property (see Heindorf and Shapiro~\cite{hs}) of boolean algebras.

\begin{definition}\label{kfn}
Given a partial order $P$ and an infinite cardinal $\ka$, 
$P$ has the $\ka$-FN iff there is a map $f\colon P\rightarrow[P]^{<\ka}$
such that for all $p\leq q$ in $P$, 
$p\leq r\leq q$ for some $r\in f(p)\cap f(q)$.  
Such an $f$ is called a $\ka$-FN map.
\end{definition}

The original Freese-Nation property, or FN, is the $\alo$-FN 
(restricted to boolean algebras), 
although it has many other equivalent definitions.  
The $\al_1$-FN is the weak Freese-Nation 
property, or WFN, which also had been studied by Heindorf and Shapiro
(again, see \cite{hs}).
Fuchino, Koppelberg, and Shelah systematically studied the
$\ka$-FN, proving an elementary submodel characterization
and a game-theoretic characterization.
They also produced examples of boolean algebras
for which the WFN is present in some model
of set theory and absent in others.  In particular,
$\mathrm{ZFC}+\neg\mathrm{CH}$ does not decide
whether $\powset{\om}$ has the WFN.

\begin{definition}[\cite{fks}]\label{ksub}
We say a subalgebra $A$ of a boolean algebra $B$
is a $\ka$-subalgebra, and write $A\leq_\ka B$,
if for all $p\in B$, the set $A\cap\downset p$
(the set of all lower bounds of $p$ in $A$) 
has cofinality less than $\ka$ 
(that is, $A\cap\downset p=A\cap\downset E$ 
for some $E\in[A]^{<\ka}$)
and $A\cap\upset p$ has coinitiality less than $\ka$.
\end{definition}
The above definition naturally generalizes to
any category of (partially) ordered structures, including
the category of partial orders.

Given a boolean algebra $B$ and an elementary submodel 
$M$ of a structure $(\bigh,\in,\leq_B)$
where $\theta$ is a sufficiently large regular cardinal,
$B\cap M$ is an elementary subalgebra of $B$.
Fuchino, Koppelberg and Shelah showed that
for regular cardinals $\ka$, 
$B$ has the $\ka$-FN iff, for all $M$
as above with $\card{M}=\ka$ and $\ka\subseteq M$,
we have $B\cap M\leq_\ka B$.

In~\cite{hs}, Heindorf proved that a boolean algebra has
the FN iff it is openly generated (a property we need
not define here), and in~\cite{sc} and \cite{sc2}, 
\scepin\ studied the Stone dual of 
open generation and generalized it from the Stone
spaces to all compacta (\ie\ compact Hausdorff spaces).
\scepin\ proved that the k-adic compacta, which are the
continuous images of openly generated compacta,
are a superclass of the dyadic compacta
(\ie\ the continuous images of powers of 2) and
that the k-adic compacta and dyadic compacta 
satisfy the same major structural theorems 
and cardinal function equations.
Bandlow~\cite{b} translated \scepin's definition 
into the language of elementary substructures: 
a compactum $X$ is openly generated iff, for a club of
countable elementary submodels $M$ of $(\bigh,\in,\mcT_X)$,
the quotient map $\pi^X_M\colon X\rightarrow X/M$ is open.
Here $\pi^X_M$ is defined by declaring $\pi^X_M(p)\not=\pi^X_M(q)$
iff there are disjoint closed neighborhoods 
$U$ of $p$ and $V$ of $q$ such that $U,V\in M$.
It follows that a boolean algebra $B$ has the FN iff
the natural quotient map from $\ult{B}$ to $\ult{B\cap M}$ 
(the latter being homeomorphic to $\ult{B}/M$)
is open for all countable elementary submodels $M$ of
 $(\bigh,\in,\leq_B)$.

In~\cite{m}, the author used elementary quotients
and a ``large submodel'' version of Bandlow's 
characterization to prove that the 
homogeneous dyadic compacta (\eg\ the compact groups)
have some strong base properties. (For example, every 
such compactum has a base such that every infinite 
subfamily's intersection has empty interior.)
To prove similar base properties for broader 
classes of compacta, the author considered 
continuous images of compacta from proper superclasses 
of the openly generated compacta. 
If we restrict from compacta to totally disconnected compacta
(\ie\ to Stone spaces), then these
superclasses can be characterized in terms of their
clopen algebras and a two-parameter family of properties 
called the \lkfn, which the author defined for
all infinite cardinals $\lm, \ka$.  

This definition of the \lkfn\ was in terms of elementary 
submodels; it was not clear whether the it could be expressed 
purely combinatorially in the spirit of Definition~\ref{kfn}.
However, a combinatorial characterization was achieved
for some important special cases. The author observed that
the $(\ka,\ka)$-FN, the $(\ka^+,\ka)$-FN, and the $\ka$-FN
are equivalent when $\ka$ is regular.  In particular,
the FN is the $(\alo,\alo)$-FN is the $(\al_1,\alo)$-FN;
the WFN is the $(\al_1,\al_1)$-FN is the $(\al_2,\al_1$)-FN.

In this paper, we systematically study the \lkfn.
In Section~\ref{sym}, we introduce a simple
combinatorial definition, prove some preservation
theorems, and show that, for all regular infinite cardinals 
$\lm>\ka$ and $\lm'>\ka'$,
the \lkfn\  implies the $(\lm',\ka')$-FN 
iff $\lm\leq\lm'$ and $\ka\leq\ka'$.
In Section~\ref{elem}, we prove that several elementary
submodel characterizations of the \lkfn,
including the original definition from~\cite{m}, 
are all equivalent to our combinatorial definition
for all regular infinite cardinals $\lm>\ka$.
In Section~\ref{cpct}, we extend classical results
of \scepin, Shapiro, and Engelking about openly
generated compacta and their continuous images 
to \laog\ compacta and their continuous images,
where the property of being $(\lm,\ka)$-openly generated 
is the natural generalization to all compacta 
of the Stone dual of the elementary submodel characterization 
of the \lkfn.

\section{Breaking a hidden symmetry of the $\ka$-FN}\label{sym}

Consider the following silly definition.

\begin{definition}\label{kkfn}
Given a partial order $P$ and an infinite cardinal $\ka$, 
$P$ has the $(\ka,\ka)$-FN iff there is a map 
$(f,g)\colon P\rightarrow[P]^{<\ka}\times[P]^{<\ka}$
such that for all $p\leq q$ in $P$, 
$p\leq r\leq q$ for some $r\in f(p)\cap g(q)$ and
$p\leq s\leq q$ for some $s\in g(p)\cap f(q)$.  
Such an $(f,g)$ is called a $(\ka,\ka)$-FN map.
\end{definition}

The definition is silly because we could declare 
$h(x)=f(x)\cup g(x)$ for all $x\in P$ to get a
$\ka$-FN map $h$ from a $(\ka,\ka)$-FN map $(f,g)$;
every $\ka$-FN map $h$ conversely yields
a $(\ka,\ka)$-FN map $(h,h)$. The $(\ka,\ka)$-FN
is just a more complicated, logically equivalent
version of the $\ka$-FN.
However, this definition has a serious purpose.
It reveals a previously unseen symmetry of the
$\ka$-FN and an obvious way to break that symmetry.

\begin{definition}\label{lkfn}
Given a partial order $P$ and 
infinite cardinals $\lm\geq\ka$, 
$P$ has the \lkfn\  iff there is a map 
$(f,g)\colon P\rightarrow[P]^{<\lm}\times[P]^{<\ka}$
such that for all $p\leq q$ in $P$, 
$p\leq r\leq q$ for some $r\in f(p)\cap g(q)$ and
$p\leq s\leq q$ for some $s\in g(p)\cap f(q)$.  
Such an $(f,g)$ is called a \lkfn\  map.
\end{definition}

Clearly, if $\lm\leq\lm'$ and $\ka\leq\ka'$, then
the \lkfn\  implies the $(\lm',\ka')$\nbd-FN,
assuming both properties are defined.
Let us show that the converse is almost true
(see Theorem~\ref{hierarchy}).
First, observe that every partial order $P$
has the $(\card{P}^+,\alo)$-FN:
just let $f(x)=P$ and $g(x)=\{x\}$ for all $x$.
In constrast, the following lemma 
gives an example of a boolean algebra of arbitrary
regular infinite cardinality $\lm$ that does not have 
the \lkfn\ for any regular $\ka<\lm$.

\begin{lemma}\label{int}
Given a regular infinite cardinal $\lm$, 
the interval algebra on $\lm$ does not have the \lkfn\
for any regular $\ka<\lm$.
\end{lemma}
\begin{proof}
Let $B$ be the interval algebra on $\lm$, \ie\ the 
algebra of finite unions of intervals of the form
$[\alpha,\beta)$ where $\alpha\leq\beta\leq\lm$.
\Wma\ we are not in the trivial case $\lm\leq\alo$.
To see that $B$ does not have the $(\lm,\ka)$-FN, 
suppose that $(f,g)\colon B\rightarrow[B]^{<\lm}\times[B]^{<\ka}$.
Let $M$ be an elementary submodel of 
$\left(H\left(\left(2^\lm\right)^+\right),\in,f\right)$
of size less than $\lm$.
Since $\lm$ is regular,
by taking the union of an appropriate
elementary chain, we can find $M$ as above
such that $\cf(\delta)=\ka$ where $\delta=\lm\cap M\in\lm$
and (consequently) $f(a)\subseteq M$ for all $a\in B\cap M$.

Let $b=[0,\delta)$.  It suffices to show that
for some $a\in B\cap M$, there is no $c\in g(b)\cap M$
satisfying $a\subseteq c\subseteq b$.
Since $\ka$ is regular and every $c\in M$ that is
a subset of $[0,\delta)$ satisfies $\sup(c)<\delta$,
there exists $\alpha<\delta$ such that $\sup(c)<\alpha$
for all $c\in g(b)\cap M$ with $c\subseteq b$.  
Hence, $a=\{\alpha\}$ is as desired.
\end{proof}

\begin{remark}
In contrast, by Proposition~2.1 of~\cite{fks}, 
the interval algebra on $\lm$ does have the $(\lm,\lm)$-FN, simply
because every infinite partial order $P$
has a $\card{P}$-FN map $q\mapsto\{p\in P: p\sqsubseteq q\}$
where $\sqsubseteq$ is a well-ordering of $P$ of
type $\card{P}$.
\end{remark}

Now, given infinite cardinals $\lm$ and $\ka$,
let $T$ be the tree $({}^{<\ka}\lm,\subseteq)$ and
let $\mcD$ consist of all subsets of $T$ of the
form $D(F)=\bigcup_{s\in F}\{t\in T:t\subsetneq s\}$
where $F$ is a finite subset of $T$.
For each $I\in\mcD$, let $Z(I)=\{p\in{}^T 2:p[I]=\{0\}\}$.
Let $B$ be the subalgebra of $\powset{{}^T 2}$ generated
by $\{Z(I):I\in\mcD\}$.

\begin{lemma}\label{lsupp}
Assuming that $\cf(\lm)\geq\ka$,
the algebra $B$ defined above has the $\ka$\nbd-FN
but does not have the $(\lm,\ka')$-FN 
for any regular infinite $\ka'<\ka$.
\end{lemma}

\begin{proof}
First, we prove that $B$ has a $\ka$-FN map $h$. 
For each $a\in B$, choose $\supp{a}$ to be some $H\in\mcD$ 
such that $a$ is in the subalgebra $A(H)$ generated by
$\{Z(I):H\supseteq I\in\mcD\}$; set $h(a)=A(\supp{a})$.
Suppose that $a,b\in B$ and $a\subseteq b$.
It follows that $\inv{\pi}[\pi[a]]\subseteq b$ where $\pi$
is the coordinate projection from ${}^T 2$ to ${}^J2$
where $J=\supp{a}\cap\supp{b}\in\mcD$.  
(To see this, suppose that $x\in {}^T 2$, $y\in a$, and
$x\restrict J=y\restrict J$. Set $z=(y\restrict\supp{a})\cup 
(x\restrict(T\setminus\supp{a})$. Since $z$ agrees with 
$y$ on $\supp{a}$, we have $z\in a\subseteq b$. Since
$x$ agrees with $z$ on $\supp{b}$, we also have $x\in b$.)
Hence, $\inv{\pi}[\pi[a]]\in h(a)\cap h(b)$
and $a\subseteq\inv{\pi}[\pi[a]]\subseteq b$.

To see that $B$ does not have the $(\lm,\ka')$-FN, 
suppose that $(f,g)\colon B\rightarrow[B]^{<\lm}\times[B]^{<\ka'}$.
Choose continuous elementary chains $(M_i:i\leq\ka')$ and
$(N_i:i\leq\ka')$ of elementary submodels of 
$\left(H\left(\left(2^{\card{T}}\right)^+\right),\in,\lm,\ka\right)$ 
such that, for all $i<\ka'$, we have $\card{M_i}<\lm$, 
$\card{N_i}\leq\om+i$, $\{M_i\}\cup\bigcup f[B\cap N_i]
\subseteq M_{i+1}$, and $(M_j)_{j\leq i}\in N_{i+1}$. 
Set $M=M_{\ka'}$, $N=N_{\ka'}$, $\delta=\min(\lm\setminus M)$, 
$\varepsilon=\min(M\cap(\delta,\lm])$, and
$C=\{\min(\lm\setminus M_i):i<\ka'\}$.
By construction, $C$ is cofinal in $\delta$, $C$ has order type $\ka'$, 
and $C\cap\gamma\in N$ for all $\gamma<\delta$. 
Moreover, $f(b)\subseteq M$ for all $b\in B\cap N$.

Let $\psi$ be the increasing bijection from $\ka'$ to $C$;
set $a=Z(D(\{\psi\}))$.
To see that $(f,g)$ is not a $(\lm,\ka')$\nbd-FN map for $B$, 
it suffices to find $b\in B\cap N$ such that
$a\subseteq b$ but there is no $c\in g(a)\cap M$
satisfying $a\subseteq c\subseteq b$. 
If $c\in M\cap\upset a$, then, by elementarity,
$c\supseteq Z(D(\{\varphi\}))$ for some $\varphi\in({}^{<\ka}\lm)\cap M$
such that $\sup(\ran(\varphi))<\varepsilon$.
Since $M\cap[\delta,\varepsilon)=\varnothing$, 
every such $\varphi$ actually satisfies $\sup(\ran(\varphi))<\delta$.
Since $\cf(\delta)>\card{g(a)}$, there exists $\beta<\ka'$
such that every $c\in g(a)\cap M\cap\upset a$
contains some $Z(D(\{\varphi\}))$ with $\sup(\ran(\varphi))<\psi(\beta)$.
Therefore, $b=Z(D(\{\psi\restrict(\beta+1)\}))$ is as desired.
\end{proof}

\begin{theorem}\label{hierarchy}
Let $(\lm,\ka)$ and $(\lm',\ka')$ each be a 
strictly decreasing pair of regular infinite cardinals.
The following are equivalent.
\begin{enumerate}
\item\label{prodord} $\lm\leq\lm'$ and $\ka\leq\ka'$.
\item\label{fnimpposet} Every poset with the \lkfn\ 
has the $(\lm',\ka')$\nbd-FN.
\item\label{fnimpbool} Every boolean algebra with the \lkfn\ 
has the $(\lm',\ka')$\nbd-FN.
\end{enumerate}
\end{theorem}
\begin{proof}
$\eqref{prodord}\Rightarrow\eqref{fnimpposet}\Rightarrow\eqref{fnimpbool}$
is trivial. We prove the contrapositive of 
$\eqref{fnimpbool}\Rightarrow\eqref{prodord}$. 

Case $\ka'<\ka$:
Let $\mu=\max(\lm,\lm')$. By Lemma~\ref{lsupp}, 
there is a boolean algebra $B$ that has
the $\ka$\nbd-FN but lacks the $(\mu,\ka')$\nbd-FN.
Therefore, $B$ has the \lkfn\ 
but lacks the $(\lm',\ka')$\nbd-FN.

Case $\ka\leq\ka'<\lm'<\lm$:
Let $B$ be the interval algebra on $\lm'$.
Since $B$ trivially has the $(\card{B}^+,\alo)$\nbd-FN
and $\card{B}=\lm'<\lm$, the algebra
$B$ also has the $(\lm,\ka)$\nbd-FN.
However, by Lemma~\ref{int}, $B$ lacks the $(\lm',\ka')$\nbd-FN.
\end{proof}

Thus, the family of \lkfn\ properties form a truly
two-dimensional hierarchy.  In particular,
for regular uncountable cardinals $\lm\geq\al_2$, 
the logical strength of the $(\lm,\alo)$-FN lies strictly between
that of the FN and $\lm$-FN, and is incomparable with
that of the $(\ka^+,\ka)$-FN for regular uncountable $\ka$ such that $\ka^+<\lm$.
In contrast, we will show in Section~\ref{elem} that the
$(\ka^+,\ka)$-FN and $(\ka,\ka)$-FN are equivalent if $\ka$ is regular.

Next, we prove some theorems about the robustness 
of our two-dimensional heirarchy. In particular, 
we extend the classical preservation (see \cite{hs}) 
of the FN by coproducts and by the exponential operation 
(see Definition~\ref{dexp}).

\begin{theorem}\label{coprod}
Given infinite cardinals $\lm\geq\ka$ and 
boolean algebras $(B_i:i\in I)$ each with the \lkfn, 
the coproduct $B=\coprod_{i\in I}B_i$ must also have the \lkfn.
\end{theorem}
\begin{proof}
This proof is a generalization of the proof of Lemma~3.1 of~\cite{m}.  
Both proofs can be seen as generalizations of a proof of the Interpolation Theorem.
For each $i\in I$, let $(f_i,g_i)$ be a \lkfn\  map for $B_i$.
For notational simplicity, assume that 
$B_i\cap B_j=\{0,1\}=\{0_{B_i},1_{B_i}\}$
for all distinct $i,j\in I$.  
Now $B$ is just the distributive
lattice freely generated by $\bigcup_{i\in I}B_i$.
Therefore, every $x\in B$ can be expressed both as
$\bigvee_{k<m(x)}\bigwedge_{i\in s(x,k)}a(x,k,i)$
and as $\bigwedge_{k<n(x)}\bigvee_{i\in t(x,k)}b(x,k,i)$
where $m(x),n(x)\in\om$, $s(x,k),t(x,k)\in[I]^{<\alo}$,
$a(x,k,i),b(x,k,i)\in B_i\setminus\{0,1\}$,
$\bigvee\varnothing=0$, and $\bigwedge\varnothing=1$.
For each $x\in B$, define $F(x)\in[B]^{<\lm}$ and $G(x)\in[B]^{<\ka}$ as the
subalgebras respectively generated by the sets $F_0(x)$ and $G_0(x)$ defined below.
(For the case $\ka=\alo$, we use the fact that finite sets generate finite boolean algebras.)

\begin{align*}
F_0(x)&=\left(\bigcup_{k<m(x)}\bigcup_{i\in s(x,k)}f_i(a(x,k,i))\right)\cup\left(\bigcup_{k<n(x)}\bigcup_{i\in t(x,k)}f_i(b(x,k,i))\right)\\
G_0(x)&=\left(\bigcup_{k<m(x)}\bigcup_{i\in s(x,k)}g_i(a(x,k,i))\right)\cup\left(\bigcup_{k<n(x)}\bigcup_{i\in t(x,k)}g_i(b(x,k,i))\right)
\end{align*}

Assume $x\leq y$ in $B$.
For every $k<m(x)$ and $l<n(y)$, 
$\bigwedge_{i\in s(x,k)}a(x,k,i)\leq\bigvee_{i\in t(y,l)}b(y,l,i)$, 
which implies $a(x,k,i_{k,l})\leq b(y,l,i_{k,l})$ 
for some $i_{k,l}\in s(x,k)\cap t(y,l)$.
Hence, $a(x,k,i_{k,l})\leq c(k,l)\leq b(y,l,i_{k,l})$ 
and $a(x,k,i_{k,l})\leq d(k,l)\leq b(y,l,i_{k,l})$ 
for some $$c(k,l)\in f_{i_{k,l}}(a(x,k,i_{k,l}))\cap g_{i_{k,l}}(b(y,k,i_{k,l}))$$
and $$d(k,l)\in g_{i_{k,l}}(a(y,k,i_{k,l}))\cap f_{i_{k,l}}(b(y,k,i_{k,l})),$$
respectively.
Therefore, we have the following inequalities.
\begin{align*}
x&\leq\bigvee_{k<m(x)}\bigwedge_{l<n(y)}a(x,k,i_{k,l})&\leq\bigvee_{k<m(x)}\bigwedge_{l<n(y)}c(k,l)&\leq\bigvee_{k<m(x)}\bigwedge_{l<n(y)}b(y,l,i_{k,l})&\leq y\\
x&\leq\bigvee_{k<m(x)}\bigwedge_{l<n(y)}a(x,k,i_{k,l})&\leq\bigvee_{k<m(x)}\bigwedge_{l<n(y)}d(k,l)&\leq\bigvee_{k<m(x)}\bigwedge_{l<n(y)}b(y,l,i_{k,l})&\leq y
\end{align*}
The middle terms of the above two inequalities witness that $(F,G)$ is a \lkfn\  map for $B$.
\end{proof}

\begin{corollary}
Let $\lm$ be an infinite cardinal.
If every cofactor of a coproduct of boolean algebras has size less than $\lm$,
then the coproduct has the $(\lm,\alo)$-FN.
\end{corollary}

In particular, the above corollary implies the 
classical result that free boolean algebras have the FN.

\begin{definition}\label{dexp}
For a given boolean algebra $B$, 
the exponential $\exp(B)$ is the clopen algebra of 
the Vietoris hyperspace $2^{\ult{B}}$ of nonempty closed subsets of $\ult{B}$.
\end{definition}

$2^{\ult{B}}$ is homeomorphic to 
the space $\mcF(B)$ of filters of $B$ topologized by 
declaring open all sets of the form $[a]=\{F\in\mcF(B): a\in F\}$
and $-[a]=\{F\in\mcF(B): a\not\in F\}$ where $a\in B$.
It has a clopen base consisting of sets of the form 
$[\bigvee_{i<n}a_i]\wedge\bigwedge_{i<n}-[-a_i]$
where $a_0,\ldots,a_{n-1}$ are nonzero and pairwise incompatible.  

In purely algebraic terms,
$\exp(B)$ is the boolean algebra freely generated by 
$\{[a]:a\in B\}$ modulo the relations 
in the sets $\{[0]=0, [1]=1\}$,
$\{[a]\leq[b]: B\ni a\leq b\in B\}$,
and $\{[a]\wedge[b]=[a\wedge b]: a,b\in B\}$.
(See~\cite{bank}.)
Note that $[a]\vee[b]\leq[a\vee b]$
and $[-a]\leq -[a]$ always hold, 
but equality fails in general.

\begin{theorem}\label{exp}
Given infinite cardinals $\lm\geq\ka$ 
and a boolean algebra $B$ with the \lkfn,
$\exp(B)$ also has the \lkfn.
\end{theorem}
\begin{proof}
Like in the proof of Theorem~\ref{coprod},
given $x\in\exp(B)$, we can express $x$ as 
$\bigvee_{k<m(x)}\bigwedge_{i\in s(x,k)}a(x,k,i)$
and $\bigwedge_{k<n(x)}\bigvee_{i\in t(x,k)}b(x,k,i)$
where $m(x),n(x)\in\om$, $s(x,k),t(x,k)\in[B]^{<\alo}$,
$a(x,k,i),b(x,k,i)\in\{\pm[i]\}\setminus\{0,1\}$,
$\bigvee\varnothing=0$, and $\bigwedge\varnothing=1$.
We will use the abbreviations
$x_k=\bigwedge_{i\in s(x,k)}a(x,k,i)$ and
$x^k=\bigvee_{i\in t(x,k)}b(x,k,i)$.

Let $(f,g)$ be a \lkfn\ map for $B$.
For each $x\in B$, first define $H(x)\in[B]^{<\alo}$ as 
the (necessarily finite) subalgebra generated by 
$\bigcup_{k<m(x)}s(x,k)\cup\bigcup_{k<n(x)}t(x,k)$.
Then define $F(x)\in[\exp(B)]^{<\lm}$ as 
the subalgebra of $\exp(B)$ generated by 
$\{[h]:h\in \bigcup f[H(x)]\}$.
Analogously define $G(x)\in[\exp(B)]^{<\ka}$.

Assume $x\leq y$ in $B$.
For every $k<m(x)$ and $l<n(y)$, we then have $x_k\leq y^l$.
Let $s^+=\{i\in s(x,k):a(x,k,i)=[i]\}$ and
$s^-=\{i\in s(x,k):a(x,k,i)=-[i]\}$;
define $t^+$ and $t^-$ analogously. Observe that
$x_k=\left[\bigwedge s^+\right]
\wedge\bigwedge_{i\in s^-}(-[i])$ and
$y^l=\left(-\left[\bigwedge t^-\right]\right)
\vee\bigvee_{i\in t^+}[i]$.
Let $U=\upset\bigwedge(s^+\cup t^-)$ and let $\eta$ be the
natural isomorphism from $\exp(B)$ to the clopen algebra
of $\mcF(B)$. There are three possibilities:
\begin{enumerate}
\item\label{notfilter} $U$ is not a (proper) filter,
\item\label{inyl} $U\in\eta(y^l)$, or
\item\label{notinyl} $U\not\in\eta(y^l)$.
\end{enumerate}

If $U$ is as in case \eqref{notfilter}, then choose 
$c(k,l)\in f\left(\bigwedge s^+\right)
\cap g\left(-\bigwedge t^-\right)$
and $d(k,l)\in g\left(\bigwedge s^+\right)
\cap f\left(-\bigwedge t^-\right)$ such that 
$c(k,l),d(k,l)\in\left[\bigwedge s^+,-\bigwedge t^-\right]$.
We then have, for each $z\in\{c(k,l),d(k,l)\}$, 
$$
x_k\leq\left[\bigwedge s^+\right]\leq[z]\leq-[-z]
\leq-\left[\bigwedge t^-\right]\leq y^l.
$$
Set $\hat c(k,l)=[c(k,l)]$ and $\hat d(k,l)=[d(k,l)]$.

If $U$ is as case \eqref{inyl}, then choose $i\in t^+$ such that 
$\bigwedge(s^+\cup t^-)\leq i$. Next choose
$c(k,l)\in f\left(\bigwedge s^+\right)
\cap g\left(\left(-\bigwedge t^-\right)\vee i\right)$
and $d(k,l)\in g\left(\bigwedge s^+\right)
\cap f\left(\left(-\bigwedge t^-\right)\vee i\right)$ such that 
$c(k,l),d(k,l)\in\left[\bigwedge s^+,
\left(-\bigwedge t^-\right)\vee i\right]$.
We then have, for each $z\in\{c(k,l),d(k,l)\}$, 
$$
x_k\leq\left[\bigwedge s^+\right]\leq[z]
\leq\left[\left(-\bigwedge t^-\right)\vee i\right]
\leq\left(-\left[\bigwedge t^-\right]\right)\vee[i]\leq y^l.
$$
Above we used the fact that if $p,q\in B$, then
$[p\vee-q]\leq[p]\vee-[q]$, which follows from
$[p\vee-q]\wedge[q]=[p\wedge q]\leq[p]$.
Set $\hat c(k,l)=[c(k,l)]$ and $\hat d(k,l)=[d(k,l)]$.

If $U$ is as case \eqref{notinyl}, then $U\not\in\eta(x_k)$, so 
choose $i\in s^-$ such that $\bigwedge(s^+\cup t^-)\leq i$. 
Next choose $d(k,l)\in f\left(\bigwedge t^-\right)
\cap g\left(i\vee\left(-\bigwedge s^+\right)\right)$
and $c(k,l)\in g\left(\bigwedge t^-\right)
\cap f\left(i\vee\left(-\bigwedge s^+\right)\right)$ such that 
$d(k,l),c(k,l)\in\left[\bigwedge t^-,i\vee(-\bigwedge s^+)\right]$.
We then have, for each $z\in\{d(k,l),c(k,l)\}$, 
$$
x_k\leq\left[\bigwedge s^+\right]\wedge-[i]
=\left[\bigwedge s^+\right]\wedge-\left[i\vee(-\bigwedge s^+)\right]
\leq-[z]\leq-\left[\bigwedge t^-\right]\leq y^l.
$$
Above we used the fact that if $p,q,r\in B$ and $p\wedge q=p\wedge r$, 
then $[p]\wedge-[q]=[p]\wedge-[r]$ because $[p]\wedge[q]=[p]\wedge[r]$.
Set $\hat d(k,l)=-[d(k,l)]$ and $\hat c(k,l)=-[c(k,l)]$.

In all three cases,
$\hat c(k,l)\in F(x)\cap G(y)$, $\hat d(k,l)\in G(x)\cap F(y)$,
and $\hat c(k,l),\hat d(k,l)\in[x_k,y^l]$. Therefore, setting 
$$c=\bigvee_{k<m(x)}\bigwedge_{l<n(y)}\hat c(k,l)\text{ and }
d=\bigvee_{k<m(x)}\bigwedge_{l<n(y)}\hat d(k,l),$$ 
we have $c\in F(x)\cap G(y)$, $d\in G(x)\cap F(y)$, 
and $c,d\in[x,y]$. Thus, $(F,G)$ is a \lkfn\ map for $\exp(B)$.
\end{proof}

For an easy third preservation theorem, we next
extend Lemma~2.7 of~\cite{fks}, which says that the 
$\ka$-FN is preserved by retracts
in the category of partial orders 
(and therefore is preserved by retracts in the
subcategory of boolean algebras).

\begin{theorem}\label{retract}
Given infinite cardinals $\lm\geq\ka$ 
partial orders $P$, $Q$ where $P$ has the \lkfn,
and order-preserving maps $i\colon P\rightarrow Q$
and $j\colon Q\rightarrow P$ such that $j\circ i=id_P$,
the partial order $Q$ must also have the \lkfn.
In particular, retracts of boolean algebras preserve
the \lkfn.
\end{theorem}
\begin{proof}
Exactly as argued in Lemma~2.7 of~\cite{fks},
given a \lkfn\ map $(f,g)$ for $P$,
$(i\circ f\circ j,i\circ g\circ j)$ is a
\lkfn\ map for $Q$.
\end{proof}

The next lemma and theorem are mostly 
due to the anonymous referee. 
The lemma yields a converse to the preservation 
by coproducts in Theorem~\ref{coprod}.
The theorem is a stronger
version of Theorem~\ref{hierarchy}.

\begin{lemma}\label{ksubalg}
If $P$ and $Q$ are partial orders, $\lm$ and $\ka$ are infinite
cardinals, $\ka$ is regular, $\lm\geq\ka$, $P\leq_\ka Q$,
and $Q$ has the \lkfn, then $P$ has the \lkfn\ too.
Hence, if a coproduct of boolean algebras has the \lkfn, then
so do all its cofactors.
\end{lemma}
\begin{proof}
The proof of the first claim of the lemma is 
the essentially the same  as its $\ka$-FN analog, 
the proof of Lemma 2.3(a) of~\cite{fks}.
Let $(f,g)$ be a \lkfn\ map for $Q$. For each $q\in Q$, 
let $L_q\in[P]^{<\ka}$ be cofinal in $P\cap\downset q$. 
For each $p\in P$, set $F(p)=\bigcup\{L_q:q\in f(p)\}$
and $G(p)=\bigcup\{L_q: q\in g(q)\}$.
If $x,y\in P$ and $x\leq y$, then we have
$x\leq z\leq y$ for some $z\in f(x)\cap g(y)$
and $x\leq w\leq y$ for some $w\in g(x)\cap f(y)$.
Choose $u\in L_z$ such that $x\leq u$;
choose $v\in L_w$ such that $x\leq v$.
We then have $x\leq u\leq y$, $u\in F(x)\cap G(y)$,
$x\leq v\leq y$, and $v\in F(x)\cap G(y)$.
Thus, $(F,G)$ is a \lkfn\ map.

For the second claim of the Lemma, we just need the
 easy classical fact that
if $B=\coprod_{i\in I}B_i$ is a coproduct of boolean algebras,
then $B_i\leq_{\alo} B$ for all $i\in I$. 
(We include $B_i$ in $B$ in the canonical way.)
Given $j\in I$, we witness $B_j\leq_{\alo}B$
with the equations $\min(B_j\cap\upset x)=x^+$ and 
$\max(B_j\cap\downset x)=x^-$ where $x^+$ and $x^-$
are defined below using the notation from the proof of 
Theorem~\ref{coprod} to again express an arbitrary $x\in B$ as 
$\bigvee_{k<m(x)}\bigwedge_{i\in s(x,k)}a(x,k,i)$
and as $\bigwedge_{k<n(x)}\bigvee_{i\in t(x,k)}b(x,k,i)$.
\begin{align*}
x^+&=\bigvee_{k<m(x)}\bigwedge\bigl\{a(x,k,i):i\in s(x,k)\cap\{j\}\bigr\}\\
x^-&=\bigwedge_{k<n(x)}\bigvee\bigl\{b(x,k,i):i\in t(x,k)\cap\{j\}\bigr\}
\end{align*}
(Recall our convention that $\bigwedge\varnothing=1$ and $\bigvee\varnothing=0$.)
\end{proof}

\begin{theorem}\label{betterhierarchy}
Given regular infinite cardinals $\mu\geq\lm>\ka$, there
is a boolean algebra $B$ such that
\begin{enumerate}
\item $B$ has the \lkfn;
\item\label{kplpl} $B$ lacks the $(\lm',\ka')$-FN 
for all regular $\lm'<\lm$ and all regular $\ka'<\lm'$;
\item\label{kpkandlpm} $B$ lacks the $(\lm',\ka')$-FN 
for all regular $\ka'<\ka$ and all regular $\lm'\leq\mu$.
\end{enumerate}
\end{theorem}
\begin{proof}
By Lemma~\ref{lsupp}, there is a boolean algebra $A$ with
the $\ka$-FN that lacks the $(\mu,\ka')$-FN for all regular $\ka'<\ka$.
Hence, $A$ has the \lkfn\ and satisfies \eqref{kpkandlpm}.
As in the proof of Thoerem~\ref{hierarchy}, 
for every regular $\lm'\in(\ka,\lm)$, 
the interval algebra $B_{\lm'}$ on $\lm'$ has the \lkfn\ 
but lacks the $(\lm',\ka')$-FN for all regular $\ka'<\lm'$.
Set $B=\coprod\bigl(\{A\}\cup\{B_{\lm'}: \ka<\cf(\lm')=\lm'<\lm\}\bigr)$. 
By Theorem~\ref{coprod}, $B$ has the \lkfn.
By Lemma~\ref{ksubalg}, $B$ satisfies \eqref{kplpl} and \eqref{kpkandlpm}.
\end{proof}

\section{Elementary submodel characterizations}\label{elem}

The proofs of Lemmas \ref{int} and \ref{lsupp} strongly
suggest that the \lkfn\ can be better understood
using elementary substructures $M$ of sufficiently large 
fragments $(\bigh,\in,\ldots)$ of the universe of sets,
just as is the case for the $\ka$-FN and for openly
generated compacta.
Moreover, restricting our consideration to $M$ such that 
$M\cap\lm$ is an initial segment of $\lm$ seems reasonable.
Choosing such an $M$ was a crucial step of the proof of 
Lemma~\ref{int}. This section confirms these conjectures.

\begin{definition}[\cite{m}, 3.16]\label{deflong}
Assuming that  $\lm$ is a regular uncountable cardinal, $\theta$ is
a sufficiently large regular cardinal, and the structure
$(\bigh,\in,\ldots)$ has signature smaller than $\lm$ 
(in this paper the signature is always finite),
a long $\lm$-approximation sequence in $(\bigh,\in,\ldots)$
is a sequence $(M_\alpha)_{\alpha<\eta}\in\bigh$ such that
\begin{itemize}
\item $M_\alpha\elemsub(\bigh,\in,\ldots)$, 
\item $\card{M_\alpha}\subseteq\lm\cap M_\alpha\in\lm$, and
\item $\lm,(M_\beta)_{\beta<\alpha}\in M_\alpha$.
\end{itemize}
\end{definition}

Let $\Omega_\lm$ denote the tree of finite sequences of
ordinals $\la\xi_i\ra_{i<n}$
which are strictly decreasing in cardinality and are all, except
possibly the last element, at least $\lm$
(more compactly, $\lm\leq\card{\xi_i}>\card{\xi_j}$ for all $\{i<j\}\subseteq n$).
There is a unique isomorphism from the ordinals
to $\Omega_\lm$ ordered lexicographically,
and its preimage of $\subseteq\restrict\Omega_\lm$ 
is a $\{\lm\}$-definable tree ordering $\trileq_\lm$ of the ordinals
which agrees with $\leq$ and has the property that
$\{\beta:\beta\trileq_\lm\alpha\}$ is finite for every ordinal $\alpha$.
The next lemma gives a very nice application of $\trileq_\lm$.

\begin{lemma}[\cite{m}, 3.17]\label{longapp}
For every $\lm$-approximation sequence $(M_\alpha)_{\alpha<\eta}\in\bigh$
in $(\bigh,\in,\ldots)$, and every $\alpha\leq\eta$, 
if we enumerate $\{\beta:\beta\trileq_\lm\alpha\}$ 
as $\{\beta_0<\cdots<\beta_m\}$, then $\beta_0=0$, 
$\beta_m=\alpha$, and 
$N_i=\bigcup\{M_\gamma:\beta_i\leq\gamma<\beta_{i+1}\}$
satisfies $\card{N_i}\subseteq N_i\elemsub(\bigh,\in,\ldots)$
for each $i<m$.
\end{lemma}

\begin{theorem}\label{equiv}
Given regular infinite cardinals $\lm>\ka$, a partial order $P$,
and a sufficiently large regular cardinal $\theta$,
the following are equivalent.
\begin{enumerate}
\item\label{bigmodel} $P\cap M\leq_\ka P$ for all 
$M\elemsub(\bigh,\in,\leq_P)$ that satisfy 
$M\cap\lm\in\lm+1$.
\item\label{bigmodelcard} $P\cap M\leq_\ka P$ for all 
$M\elemsub(\bigh,\in,\leq_P)$ that satisfy 
$\card{M}\cap\lm\subseteq M\cap\lm\in\lm+1$.
\item\label{smallmodel} $P\cap M\leq_\ka P$ for all 
$M\elemsub(\bigh,\in,\leq_P)$ that satisfy $M\cap\lm\in\lm$.
\item\label{smallmodelcard} $P\cap M\leq_\ka P$ for all 
$M\elemsub(\bigh,\in,\leq_P)$ that satisfy 
$\card{M}\subseteq M\cap\lm\in\lm$.
\item\label{club} $A\leq_\ka P$ for all $A$ in some club 
$\mcE\subseteq[P]^{<\lm}$.
\item\label{fnmap} $P$ has the \lkfn.
\end{enumerate}
\end{theorem}
\begin{remark}\ 
\begin{itemize}
\item In general, for regular $\lm$ and $M\elemsub(\bigh,\in,\ldots)$, 
$M\cap\lm\in\lm+1$ iff $[\bigh]^{<\lm}\cap M\subseteq[M]^{<\lm}$.
\item Item~\eqref{bigmodelcard} in the theorem is 
the original definition of the \lkfn\ from~\cite{m} (4.1).\
\end{itemize}
\end{remark}
\begin{proof}
We will prove the following two overlapping circles of implications.
\begin{align*}
\eqref{bigmodel}\Rightarrow
\eqref{smallmodel}\Rightarrow
\eqref{smallmodelcard}\Rightarrow
\eqref{club}&\Rightarrow
\eqref{bigmodel}
\\
\eqref{club}&\Rightarrow
\eqref{bigmodel}\Rightarrow
\eqref{bigmodelcard}\Rightarrow
\eqref{fnmap}\Rightarrow
\eqref{club}
\end{align*}

$\eqref{bigmodel}\Rightarrow
\eqref{smallmodel}\Rightarrow
\eqref{smallmodelcard}$: Trivial.

$\eqref{smallmodelcard}\Rightarrow\eqref{club}$:
If $\card{P}<\lm$, then $\mcE=\{P\}$ works. So, assume $\card{P}\geq\lm$.
Choose a well-ordering $\sqsubseteq$ of $\bigh$; now 
$(\bigh,\in,\leq_P,\sqsubseteq)$ has a Skolem hull operator $\Hull$.
Let $\mcE$ be the set of all $A\in[P]^{<\lm}$ 
that satisfy $\card{A}\subseteq \Hull(A)\cap\lm\in\lm$ 
and $A=P\cap\Hull(A)$. 
Since $\lm$ is regular, $\mcE$ is closed
with respect to ascending unions of length less than $\lm$. 
Moreover, starting from any $X_0\in[P]^{<\lm}$, we can expand $X_0$
to $X=\bigcup_{n<\om}X_n$ where 
$X_{n+1}=P\cap\Hull(X_n\cup\card{X_n}\cup\sup(\lm\cap \Hull(X_n)))$.
We have $X\in\mcE$ because there are Skolem terms for $\varphi$ and 
$\inv{\varphi}$ where $\varphi$ is the $\sqsubseteq$\nbd-least 
injection from $\lm$ to $P$. Thus, $\mcE$ is unbounded.
Finally, if $A\in\mcE$, then 
$\Hull(A)\elemsub(\bigh,\in,\leq_P,\sqsubseteq)$, so  
\eqref{smallmodelcard} implies that $A=P\cap \Hull(A)\leq_\ka P$.

$\eqref{club}\Rightarrow\eqref{bigmodel}$: 
Assume $\mcE\subseteq[P]^{<\lm}$ is as in \eqref{club},
$M\elemsub(\bigh,\in,\leq_P)$, and $M\cap\lm\in\lm+1$.
We want to show that $P\cap M\leq_\ka P$.
By elementarity, there exist $\lm',\ka',\mcE'\in M$ 
such that $\lm'>\ka'$ are regular infinite cardinals, 
$\mcE'$ is a club subset of $[P]^{<\lm'}$,
and $A\leq_{\ka'}P$ for all $A\in\mcE'$.
First, we show that we can choose $\lm'\leq\lm$ and $\ka'\leq\ka$.
If $\ka\in M$, then we can choose $\ka'=\ka$ and
next choose $\lm'$ to be least possible,
which guarantees $\lm'\leq\lm$.
If $\lm\in M$, then we can choose $\lm'=\lm$ and
next choose $\ka'$ to be least possible,
which guarantees $\ka'\leq\ka$.
Suppose we are in the remaining case, that
$M\cap\lm\leq\ka$ and $\lm\not\in M$.
If $\theta\cap M=\ka$, then $\ka'<\lm'<\ka<\lm$.
If $\mu=\min(\theta\cap M\setminus\ka)$, then
$\mu>\lm$, so we can choose $\ka'<\lm'<\mu$, so
again $\ka'<\lm'<\ka<\lm$.

Now, safely assuming that $\lm'\leq\lm$ and $\ka'\leq\ka$,
it suffices to show that $P\cap M\leq_{\ka'}P$. Fix $q\in P$. 
By symmetry, we need only prove that $\cf(M\cap\downset q)<\ka'$.
Set $\chi=\left(2^{<\theta}\right)^+$ and choose $N_0$ such that
$N_0\elemsub(H(\chi),\in,\leq_P,q,M,\ka',\mcE')$.
Since $\lm'\leq\lm$, we have $M\cap\lm'\in\lm'+1$, 
so we may choose $N_0$ such that 
$\card{N_0}\subseteq N_0\cap\lm'\in\lm'$
by taking $N_0$ to be the union of an appropriate elementary chain.
By elementarity, $M\cap N_0\elemsub(\bigh,\in,\leq_P,\mcE')$ 
and $M\cap N_0\cap\lm'=N_0\cap\lm'\in\lm'$.
Hence, $P\cap M\cap N_0\in\mcE'$; hence, $P\cap M\cap N_0\leq_{\ka'} P$.
Let $D_0$ be a cofinal subset of 
$M\cap N_0\cap\downset q$ of size less than $\ka'$.

Starting with $N_0$, we now form a continuous elementary chain
of length $\ka'+1$.
Given $\alpha<\ka'$ and $(N_\alpha,D_\alpha)$ having
the above properties of $(N_0,D_0)$, choose $(N_{\alpha+1}, D_{\alpha+1})$
to have the above properties of $(N_0,D_0)$ while also
satisfying $(D_i)_{i\leq\alpha}\in N_{\alpha+1}$ and $D_\alpha\subseteq D_{\alpha+1}$.
At limit stages $\alpha\leq\ka'$, let $N_\alpha=\bigcup_{i<\alpha}N_i$
and $D_\alpha=\bigcup_{i<\alpha}D_i$, noting
that the above properties of $(N_0,D_0)$ are preserved by
unions of chains of length at most $\ka'$, except
possibly that $\card{D_{\ka'}}=\ka'$. 
In particular, $P\cap M\cap N_{\ka'}\in\mcE'$, so
$\cf(M\cap N_{\ka'}\cap\downset q)<\ka'$. 
Hence, $D_\alpha$ is cofinal in $M\cap N_{\ka'}\cap\downset q$ 
for some $\alpha<\ka'$.
Since $D_\alpha\in N_{\ka'}$, $D_\alpha$ is, by elementarity, 
also cofinal in $M\cap\downset q$.
Thus, $\cf(M\cap\downset q)<\ka'$ as desired.

$\eqref{bigmodel}\Rightarrow\eqref{bigmodelcard}$: Trivial.

$\eqref{bigmodelcard}\Rightarrow\eqref{fnmap}$:
Let $\card{P}=\mu\geq\lm$ and let $(M_i)_{i<\mu}$
be a long $\lm$-approximation sequence in 
$(\bigh,\in,\leq_P)$.
Observe that $P\subseteq\bigcup_{i<\mu}M_i$.
Assuming that $i<\mu$ and we have defined 
$(f,g)\restrict\left(P\cap\bigcup_{j<i}M_j\right)$
to be a \lkfn\ map on $P\cap\bigcup_{j<i}M_j$, 
it suffices to extend the definition 
to get a \lkfn\ map on $P\cap\bigcup_{j\leq i}M_j$.

Let $N_0,\ldots,N_{m-1}$ be as in Lemma~\ref{longapp};
observe that $\bigcup_{j<i}M_j=\bigcup_{j<m}N_j$.
Given $x\in P\cap M_i\setminus\bigcup_{j<m}N_j$,
let $A_j\in[P\cap N_j]^{<\ka}$ include a cofinal
subset of $N_j\cap\downset x$ and a coinitial
subset of $N_j\cap\upset x$, for each $j<m$.
Finally, set $f(x)=(P\cap M_i)\cup\bigcup_{j<m}\bigcup\{f(z):z\in A_j\}$
and $g(x)=\{x\}\cup\bigcup_{j<m}\bigcup\{g(z):z\in A_j\}$.
It is easy to check that this definition works.
For example, if $x\leq y\in P\cap N_j$, 
then $x\leq z\leq y$ for some $z\in A_j$;
since $z,y\in N_j$, 
we have $z\leq w\leq y$ for some $w\in f(z)\cap g(y)$;
since $f(z)\subseteq f(x)$, 
we have $x\leq w\leq y$ and $w\in f(x)\cap g(y)$.

$\eqref{fnmap}\Rightarrow\eqref{club}$: 
Let $(f,g)$ be a \lkfn\ map.
Let $\mcE$ be the club set of all $X\in[P]^{<\lm}$ 
satisfying $X\supseteq\bigcup f[X]$. Fix $A\in\mcE$.
For each $p\in P$ and $q\in A\cap\upset p$,
there exists $r\in g(p)\cap f(q)$ such that $p\leq r\leq q$.
Since $f(q)\subseteq A$, $g(p)\cap A$ includes a coinitial subset of 
$M\cap\upset p$; by symmetry, $g(p)\cap A$ includes a cofinal subset of 
$M\cap\downset p$. Since $\card{g(p)}<\ka$, $A\leq_\ka P$.
\end{proof}

\begin{corollary}\label{kkpk}
For all regular infinite $\ka$, the $\ka$-FN is equivalent to the $(\ka^+,\ka)$-FN.
\end{corollary}
\begin{proof}
By Proposition~3.1 of~\cite{fks}, for all regular infinite $\ka$,
the $\ka$-FN is equivalent to item~\eqref{club} of Theorem~\ref{equiv}
with $\lm=\ka^+$.  

For a self-contained proof, first note that 
the $(\ka,\ka)$-FN trivially implies the $(\ka^+,\ka)$-FN, 
so we just need to modify the proof of 
$\eqref{bigmodelcard}\Rightarrow\eqref{fnmap}$ of Theorem~\ref{equiv}
so as to use an additional hypothesis $\lm=\ka^+$ 
to produce a $\ka$-FN map.
To accomplish this, replace the inductive definitions of $f$ and $g$
with a similar inductive definition of $f$ alone where now
$f(x)=\{z: z\sqsubseteq x\}\cup\bigcup_{j<m}\bigcup\{f(z):z\in A_j\}$
where $\sqsubseteq$ is a well-ordering of $P\cap M_i$
of length at most $\ka$.
\end{proof}

\section{Consequences for compacta}\label{cpct}

Let us review a few basic facts about elementary quotients.
Assume $X$ is a compactum and $M\elemsub\la\bigh,\in,\mcT_X\ra$.
We then have the following.
\begin{itemize}
\item $\pi^X_M(p)=\pi^X_M(q)$ iff $f(p)=f(q)$ for all 
continuous real-valued functions $f\in M$.
\item If $U\in M$ is functionally open, then $\pi^X_M[U]$ is functionally open and
$U=\inv{\left(\pi^X_M\right)}\left[\pi^X_M[U]\right]$.
\item Moreover, $X/M$ has a base consisting of sets of
the form $\pi^X_M[U]$ where $U$ is functionally open and $U\in M$.
\item $\pi^X_M[X\cap M]$ is dense in $X/M$.
\end{itemize}

Lemma~4.11 of~\cite{m} shows that Bandlow's characterization of ``openly generated,''
that $\pi^X_M$ is open for a club of $M\in[\bigh]^{<\al_1}$,
is equivalent to $\pi^X_M$ being open for all $M\elemsub(\bigh,\in,\mcT_X)$,
provided $\theta$ is sufficiently large.
(We will prove a more general result in Theorem~\ref{ogequiv}.)
Given an arbitrary regular uncountable cardinal $\lm$,
a simple generalization of Bandlow's property is to declare a compactum to be
\laog\ if $\pi^X_M$ is open for a club of $M\in[\bigh]^{<\lm}$,
which is likewise equivalent to $\pi^X_M$ being open 
for all $M\elemsub(\bigh,\in,\mcT_X)$ satisfying $M\cap\lm\in\lm+1$, 
provided $\theta$ is sufficiently large.
This generalization also appears to be the correct generalization,
as shown by the theorems of this section.

Without explicitly defining the notion ``\lkog,'' 
Theorem 4.17 of~\cite{m} used a naively generalized
hypothesis to deduce some order-theoretic base properties.
The hypothesis was that, for some infinite cardinal $\mu$, $\lm=\mu^+$, $\ka=\cf(\mu)$, and $\pi^X_M[U]$ is the intersection of fewer than $\ka$\nbd-many
open sets, for all open $U\subseteq X$ and all $M\elemsub(\bigh,\in,\mcT_X)$
that satisfy $\card{M}\cap\lm\subseteq M\cap\lm\in\lm+1$.
Among the conclusions was that every continuous image $Y$ of $X$ has a $\pi$\nbd-base 
such that every $\ka$-sized subfamily's intersection has empty interior;
if also $\pi\character{y,Y}=\weight{Y}$ for all $y\in Y$, then
``$\pi$\nbd-base'' can be strengthened to ``base.''

Reading through the proof, it is not hard to check that 
the universal quantification over all open $U$ is stronger than necessary;
it suffices for $X$ to have a base $\mcB$ of functionally open sets such
that $U$ is as above for every $U\in\mcB$.
Moreover, from the perspective of choosing a good definition
of \lkog, the latter, narrower quantification is preferable, as shown by
the following example.

\begin{example}\label{dblarrow}
Let $D$ be the double-arrow space, that is, 
the lexicographically ordered space $([0,1]\times\{0,1\})\setminus\{\la 0,0\ra,\la 1,1\ra\}$.
This space is a totally disconnected compactum 
and its clopen algebra is isomorphic to the real interval algebra,
which has the WFN.  
Let $X=D^2$.  Since $\Clop(X)$ is isomorphic
to the coproduct $\Clop(D)\coprod\Clop(D)$, 
$\Clop(X)$ also has the WFN.
Therefore, $X$ should be $(\al_2,\al_1)$-openly generated for any good
definition of $(\al_2,\al_1)$-openly generated.

In the closed subspace $C=\{\la\la a,i\ra,\la b,j\ra\ra\in X:a+b=1\}$,
the set $I$ of isolated points is $\{\la\la a,i\ra,\la b,j\ra\ra\in C:i=j\}$.
Observe that if $L=C\setminus I$, then $L$ is homemorphic to $D$.
Let $U=(X\setminus C)\cup I$, which is open in $X$.
Let $M\elemsub(\bigh,\in)$ and $\card{M}\subseteq M\cap\om_2\in\om_2$.
We will show that $\neg CH$ implies that $\pi^X_M[U]$ is not a $G_\delta$ set.

We will prove the contrapositive; assume $\pi^X_M[U]$ is $G_\delta$.
First, observe that $\left(\pi^X_M\right)^{-1}[\pi^X_M[U]]=X\setminus(L\cap M)$
because a continuous real valued function from $M$ can separate the points
$\la\la a,i\ra,\la b,j\ra\ra$ and $\la\la a,i'\ra,\la b,j'\ra\ra$
iff $\la i,j\ra\not=\la i',j'\ra$ and $a,b\in M$.
Hence, $L\cap M$ is $F_\sigma$, so $D\cap M$ is $F_\sigma$.
Applying the canonical finite-to-one map from $D$ onto $[0,1]$, we
see that $[0,1]\cap M$ is $F_\sigma$.  
Since every uncountable Borel subset of $[0,1]$ contains a perfect set, 
$M$ contains a perfect set, so CH holds.
\end{example}

\begin{question}
In the presence of CH, the above $X$ has weight $\al_1$, so
$\pi^X_M$ becomes a trivial homeomorphism.
Is there a GCH analog of Example~\ref{dblarrow}?
\end{question}

\begin{definition}\label{dlkog}\
\begin{itemize}
\item An $(M,\ka)$-base of a compactum $X$ is
a base $\mcB$ of $X$ consisting only of functionally open sets $U$
for which $\pi^X_M[U]$ is the intersection of fewer than
$\ka$-many open sets.
\item A \lkb\ of a compactum $X$ is
a base $\mcB$ that is an $(M,\ka)$-base for all $M\elemsub(\bigh,\in,\mcB)$ 
satisfying $\card{M}\cap\lm\subseteq M\cap\lm\in\lm+1$.
\item A strong \lkb\ (respectively, strong $(M,\ka)$-base) of a compactum $X$ is
a \lkb\ (respectively, $(M,\ka)$-base) $\mcB$ of $X$ that satisfies
$U\cup V\in\mcB$ for all $U,V\in\mcB$.
\item A compactum $X$ is \lkog\ if it has a \lkb.
\item A compactum is \lka\ if it is
the continuous image of a \lkog\ compactum.
\end{itemize}
\end{definition}

\begin{remark}\
\begin{itemize}
\item If $\mcB$ is a \lkb,
then $\{\bigcup\mcF:\mcF\in[\mcB]^{<\alo}\}$ is a strong \lkb.
\item If $X$ has a $(\lm,\alo)$-base,
then $\pi^X_M$ is open for all $M\elemsub(\bigh,\in,\mcT_X)$ 
satisfying $M\cap\lm\in\lm+1$.
The converse is trivially true.
\end{itemize}
\end{remark}

\begin{theorem}\label{ogequiv}
Given regular infinite cardinals $\lm>\ka$, a compactum $X$,
a base $\mcB$ of $X$, and a sufficiently large regular $\theta$,
the following are equivalent.
\begin{enumerate}
\item\label{bigmbase} $\mcB$ is an $(M,\ka)$-base for all 
$M\elemsub(\bigh,\in,\mcB)$ satisfying $M\cap\lm\in\lm+1$.
\item\label{bigmbasecard} $\mcB$ is an $(M,\ka)$-base for all 
$M\elemsub(\bigh,\in,\mcB)$ satisfying 
$\card{M}\cap\lm\subseteq M\cap\lm\in\lm+1$.
\item\label{smallmbase} $\mcB$ is an $(M,\ka)$-base for all 
$M\elemsub(\bigh,\in,\mcB)$ satisfying $M\cap\lm\in\lm$.
\item\label{smallmbasecard} $\mcB$ is an $(M,\ka)$-base for all 
$M\elemsub(\bigh,\in,\mcB)$ satisfying $\card{M}\subseteq M\cap\lm\in\lm$.
\item\label{clubmbase} $\mcB$ is an $(M,\ka)$-base for club-many 
$M\in[\bigh]^{<\lm}$.
\end{enumerate}
If $X$ is also totally disconnected, then $X$ is \lkog\ iff
$\Clop(X)$ has the \lkfn.
\end{theorem}
\begin{proof}
Trivially, 
$\eqref{bigmbase}\Rightarrow
\eqref{bigmbasecard}\Rightarrow
\eqref{smallmbasecard}$
and
$\eqref{bigmbase}\Rightarrow
\eqref{smallmbase}\Rightarrow
\eqref{smallmbasecard}$.
Therefore, it suffices to show that
$\eqref{smallmbasecard}\Rightarrow
\eqref{clubmbase}\Rightarrow
\eqref{bigmbase}$.

$\eqref{smallmbasecard}\Rightarrow\eqref{clubmbase}$:
Proceed similarly to the proof of 
$\eqref{smallmodelcard}\Rightarrow\eqref{club}$
of Theorem~\ref{equiv}.
Let $\mcE$ be the set of all $M\elemsub(\bigh,\in,\mcB)$
satisfying $\card{M}\subseteq M\cap\lm\in\lm$.
Every union of a chain in $\mcE$ of length less than $\lm$
is itself in $\mcE$, so $\mcE$ is closed in $[\bigh]^{<\lm}$.
Moreover, for every $A\in[\bigh]^{<\lm}$,
the union of an appropriate elementary chain 
will both contain $A$ and be in $\mcE$.
Thus, $\mcE$ is a club subset of $[\bigh]^{<\lm}$.

$\eqref{clubmbase}\Rightarrow\eqref{bigmbase}$:
Let $M\elemsub(\bigh,\in,\mcB)$ and $M\cap\lm\in\lm+1$.
Now observe that increasing $\theta$ preserves \eqref{clubmbase}:
assuming $\theta$ is already large enough that $\mcB\in\bigh$,
if $\mcE\subseteq[\bigh]^{<\lm}$ is a club and $\nu>\theta$,
then $\pi^X_N=\pi^X_{N\cap\bigh}$ for all $N\elemsub(H(\nu),\in,\mcB)$,
and club\nbd-many $N\in[H(\nu)]^{<\lm}$ satisfy 
$N\elemsub(H(\nu),\in,\mcB)$ and $N\cap\bigh\in\mcE$.
Therefore, \wma\ that $\theta$ is sufficiently large 
for us to apply elementarity to produce 
regular infinite cardinals $\ka'<\lm'\leq\theta'$
and a club $\mcD\subseteq[H(\theta')]^{<\lm'}$ such that 
$\mcD,\ka',\lm'\in M$, $\mcB\in H(\theta')$, and
$\mcB$ is a $(N,\ka')$-base for all $N\in\mcD$.
If $\lm'\cap M\in\lm'+1$, then $M\cap H(\theta')\in\mcD$
and $\pi^X_M=\pi^X_{M\cap H(\theta')}$, 
so $\mcB$ is an $(M,\ka')$\nbd-base.
Just as in the proof of $\eqref{club}\Rightarrow\eqref{bigmodel}$ 
of Theorem~\ref{equiv}, we can always choose $\lm'\leq\lm$ and
$\ka'\leq\ka$, so $\mcB$ is an $(M,\ka)$\nbd-base.
Thus, $\eqref{clubmbase}\Rightarrow\eqref{bigmbase}$.

Now assume that $X$ is totally disconnected, 
$M\elemsub(\bigh,\in,\mcT_X)$,
$A\in \mcA=\Clop(X)$, and $M\cap\lm\in\lm$.
First, note that any strong $(M,\ka)$-base $\mcB$
includes all the nonempty clopen sets, and that
if $\pi^X_M[A]$ and $\pi^X_M[X\setminus A]$ are 
each the intersection of fewer than $\ka$-many open sets,
then, by compactness, they each have clopen neighborhood bases
of size less than $\ka$; hence, $M\cap\upset_\mcA A$ has
coinitiality less than $\ka$ and $M\cap\downset_\mcA A$ has
cofinality less than $\ka$. 
Conversely, if $M\cap\upset_\mcA C$ has coinitiality less than $\ka$ 
for every clopen $C$, 
then $\Clop(X)$ is a strong $(M,\ka)$-base.
Thus, $X$ is \lkog\ iff $\Clop(X)$ has the \lkfn.
\end{proof}

\begin{corollary}
A compactum is openly generated iff it is $(\al_1,\alo)$-openly generated.
\end{corollary}

For all regular infinite cardinals $\lm>\ka$,
the next theorem consists of Stone duals of some results from Section~\ref{sym},
naturally generalized from the totally disconnected compacta to all compacta.
\begin{theorem}
Assume $\lm$ and $\ka$ are infinite cardinals and $\lm\geq\ka$.
\begin{enumerate}
\item\label{smallw} A compactum with weight less than $\lm$ (see Def.~\ref{cardf}) is \laog.
\item\label{prod} Products of \lkog\ compacta are \lkog.
\item\label{hyper} The Vietoris hyperspace of a \lkog\ compactum is \lkog.
\end{enumerate}
\end{theorem}
\begin{proof}
\eqref{smallw}: If $\mcB$ is a base of a compactum $X$, 
$\card{\mcB}<\lm$, $M\elemsub(\bigh,\in,\mcB)$, and $M\cap\lm\in\lm+1$,
then $\mcB\subseteq M$, so $\pi^X_M$ is $id_X$
(modulo identifying points with singleton equivalence classes), 
so $\pi^X_M$ is open.

\eqref{prod}: Given $X=\prod_{i\in I}X_i$, assume each $X_i$ has a \lkb\ $\mcB_i$,
and let $\mcB$ be the base of $X$ consisting of all open boxes $U=\bigcap_{i\in s}\inv{\pi_i}[U_i]$
where $s\in[I]^{<\alo}$ and each $U_i$ is in $\mcB_i$ and is not $\varnothing$ nor $X_i$.
Fix $M\elemsub(\bigh,\in,\mcB)$ such that 
$\card{M}\cap\lm\subseteq M\cap\lm\in\lm+1$; 
we then have $(\mcB_i:i\in I)\in M$.
Given $U$ as above, $U$ is functionally open, so
it suffices to show that $\pi^X_M[U]$ is the intersection of fewer than $\ka$-many open sets.
If $p,q\in X$, then $\pi^X_M(p)\not=\pi^X_M(q)$ iff $\pi^{X_i}_M(p(i))\not=\pi^{X_i}_M(q(i))$
for some $i\in I\cap M$.  
Therefore, $\pi^X_M[U]=\bigcap_{i\in s\cap M}\inv{\pi_i}[\pi^{X_i}_M[U_i]]$,
which is an intersection of fewer than $\ka$-many open sets because 
each $\pi^{X_i}_M[U_i]$ is an intersection of fewer than $\ka$-many open sets.

\eqref{hyper}: Let $\mcB$ be a \lkb\ of $X$; let $Y=2^X$.
Given $E\subseteq X$, let $[E]=\{F\in Y:F\subseteq E\}$.
By definition, $\{[O]: O\text{ open}\}\cup\{Y\setminus[C]: C\text{ closed}\}$
is a subbase of $Y$.
Observe that $\closure{[E]}=\Bigl[\closure{E}\Bigr]$.
Also, if $\mcE\subseteq\powset{X}$, then $[\bigcap\mcE]=\bigcap_{E\in\mcE}[E]$.
By compactness, if $\mcU$ is a family of open subsets of $X$, then 
$[\bigcup\mcU]=\bigcup\{[\bigcup\mcF]:\mcF\in[\mcU]^{<\alo}\}$.
Therefore, if $U\subseteq X$ is functionally open, then,
since being functionally open is equivalent to
being a union $\bigcup_{n<\om}E_n$ where the closure of $E_n$
is contained in the interior of $E_{n+1}$,
$[U]$ is also functionally open.

Given $\mcF\in[\mcB]^{<\alo}$, 
let $\la\mcF\ra=[\bigcup\mcF]\setminus\bigcup_{U\in\mcF}[X\setminus U]$,
which is the set of all nonempty closed subsets of $\bigcup\mcF$ that
intersect every element of $\mcF$.
By compactness,
$\mcA=\{\la\mcF\ra:\mcF\in[\mcB]^{<\alo}\}$ is a base of $Y$.
Morover, since every $U\in\mcB$ is functionally open, 
every $\la\mcF\ra\in\mcA$ is functionally open.
Let $M\elemsub(\bigh,\in,\mcA)$ and 
$\card{M}\cap\lm\subseteq M\cap\lm\in\lm+1$.
It suffices to show that  $\mcA$ is a $(M,\ka)$-base of $Y$.

For each $E\subseteq X$, set $E_M=\inv{\left(\pi^X_M\right)}[\pi^X_M[E]]$;
likewise define $E_M$ for each $E\subseteq Y$.
Observe that $\mcB\in M$, so $\mcB$ is an $(M,\ka)$-base of $X$, so
if $\la\mcF\ra\in\mcA$, then $\la\mcF_M\ra=\la\{U_M:U\in\mcF\}\ra$ 
is the intersection of fewer than $\ka$-many open sets:
if $\mcF_M=\{\bigcap_{\alpha<\beta(i)}O^i_\alpha:i<n\}$ where $\beta(i)<\ka$,
then $\la\mcF_M\ra=\bigcap\{\la\{O^i_{t(i)}:i<n\}\ra:t\in\prod_{i<n}\beta(i)\}$.
Moreover, since $\pi^Y_M$ is a quotient map, 
$\pi^Y_M[\la\mcF\ra]$ is the intersection of fewer than $\ka$-many open sets
iff $\la\mcF\ra_M$ is.
Therefore, it suffices to show that $\la\mcF\ra_M=\la\mcF_M\ra$ for all
$\la\mcF\ra\in\mcA$.
Every $\la\mcF\ra\in\mcA$ is functionally open, so it is a union
$\bigcup_{i<\omega}\closure{\la\mcF_i\ra}$ where $\la\mcF_i\ra\in\mcA$.
Therefore, it suffices to prove that 
$\la\mcF\ra_M\subseteq\la\mcF_M\ra\subseteq\left(\closure{\la\mcF\ra}\right)_M$
for every $\la\mcF\ra\in\mcA$.

Fix $\la\mcF\ra\in\mcA$. First suppose that $H\in\la\mcF_M\ra$.  
To prove that $H\in\left(\closure{\la\mcF\ra}\right)_M$,
it suffices to show that for every $\la\mcG\ra\in\mcA\cap M$,
if $H\in\la\mcG\ra$, then $\la\mcF\ra\cap\la\mcG\ra\not=\varnothing$.
(Indeed, if $K_{\la\mcG\ra}\in\la\mcF\ra\cap\la\mcG\ra$
whenever $H\in\la\mcG\ra\in\mcA\cap M$, then we have a net $\vec{K}$
which has one or more cluster points $L$, all of which are
in $\closure{\la\mcF\ra}$ and satisfy $\pi^Y_M(L)=\pi^Y_M(H)$.)
So, suppose $H\in\la\mcG\ra\in\mcA\cap M$.  
For each $U\in\mcF$, choose $a_U\in H\cap U_M$ and
then choose $a_U'\in U$ such that $\pi^X_M(a_U')=\pi^X_M(a_U)$.
For each $V\in\mcG$, choose $b_V\in H\cap V\subseteq\bigcup(\mcF_M)=(\bigcup\mcF)_M$
and then choose $b'_V\in\bigcup\mcF$ such that $\pi^X_M(b_V')=\pi^X_M(b_V)$.
Set $K=\{a_U':U\in\mcF\}\cup\{b_V':V\in\mcG\}$.
Fix $U\in\mcF$ and $V\in\mcG$. 
We have $a_U'\in U\subseteq\bigcup\mcF$ 
and $b_V'\in\bigcup\mcF$, 
so $K\subseteq\bigcup\mcF$;
we have $a_U'\in H_M\subseteq\bigcup\mcG_M=\bigcup\mcG$ 
and $b_V'\in V_M=V\subseteq\bigcup\mcG$, 
so $K\subseteq\bigcup\mcG$;
we have $a_U'\in U$, so $K\cap U\not=\varnothing$;
we have $b_V'\in V_M=V$, so $K\cap V\not=\varnothing$.
Thus, $K\in\la\mcF\ra\cap\la\mcG\ra$ as desired.

Now suppose that $H\not\in\la\mcF_M\ra$.  
All that remains is to show that $H\not\in\la\mcF\ra_M$.
Fix $K\in\la\mcF\ra$; it suffices to show that $H$ and $K$
have disjoint closed neighborhoods in $M$.
By definition, at least one of two cases occurs: 
$H\not\subseteq\bigcup\mcF_M$ or
$H\cap U_M=\varnothing$ for some $U\in\mcF$.
In the first case, we may choose $p\in H\setminus K_M$
and then, by compactness, neighborhoods $V,W\in M$
of $p$ and $K$, respectively, 
such that $\closure{V}\cap\closure{W}=\varnothing$.
For such $V$ and $W$, we have
$H\in Y\setminus[X\setminus V]$, $K\in[W]$,
and $\closure{Y\setminus[X\setminus V]}\cap\closure{[W]}=\varnothing$,
so $H$ and $K$ have disjoint closed neighborhoods in $M$.
In the second case, we may choose $p\in K\setminus H_M$
and then, by compactness, 
neighborhoods $V,W\in M$ of $p$ and $H$, respectively,
such that $\closure{V}\cap\closure{W}=\varnothing$.
For such $V$ and $W$, we have
$K\in Y\setminus[X\setminus V]$, $H\in[W]$, 
and $\closure{Y\setminus[X\setminus V]}\cap\closure{[W]}=\varnothing$,
so $H$ and $K$ have disjoint closed neighborhoods in $M$.
\end{proof}

\begin{remark} \scepin~\cite{sc} proved that openly generated compacta
have openly generated hyperspaces using a completely different
argument based on his characterization of open generation
as the existence of a $k$-metric, a notion of distance 
between points and regular closed sets.
\end{remark}

Shapiro used hyperspaces to demonstrate 
that not all openly generated compacta are 
continuous images of products of 
spaces with small weight.  In particular,
not all openly generated compacta are dyadic:

\begin{theorem}[Shapiro~\cite{sa}] 
For all infinite successor cardinals $\lm$,
if $X$ is a compactum with weight greater than $\lm$
(\eg\ the openly generated ${}^{\lm^+}2$), then
the Vietoris hyperspace $2^X$ is not a continuous image of 
any product of compacta whose factors all have weight less than $\lm$.
\end{theorem}

Moreover, not all dyadic compacta are open generated;
Engelking proved this in~\cite{e}, and Fuchino, Koppelberg,
and Shelah noted the Stone dual of this in terms of
the $\ka$-FN in Proposition 7.6 of~\cite{fks}.
The proof from~\cite{fks} easily generalized to yield the following.

\begin{theorem}[\cite{e}; \cite{fks}]
For any two infinite cardinals $\lm>\ka$,
the quotient of ${}^\lm 2$ formed by identifying $(0:i<\lm)$ and
$(1:i<\lm)$  is not \lkog.
\end{theorem}

Next we consider some consequences for cardinal functions
in topology.

\begin{definition}\label{cardf}\
\begin{itemize}
\item A base of a compactum is a family of nonempty open sets $\mcB$ 
such that for every closed subset $K$ and neighborhood $U$ of $K$,
$K\subseteq\bigcup\sigma\subseteq U$ for some finite $\sigma\subseteq\mcB$.
(This agrees with the usual definition of a topological base.)
\item A base of a boolean algebra $B$ is a subset $J$ such that
every element of $B$ is a finite join of elements from $J$.
\item The weight $\weight{X}$ of a space or boolean algebra $X$ 
is the least $\ka\geq\om$ for which $X$ has a base of size $\leq\ka$.
(Note that $\weight{B}=\card{B}$ for all infinite boolean algebras $B$.)
\item A local $\pi$-base at $p\in X$ is a family $\mcB$ of nonempty open
sets such that every neighborhood of $p$ includes some element of $\mcB$
as a subset.
\item A local $\pi$-base at an ultrafilter $U$ of a boolean algebra $B$
is a subset $S$ of $B\setminus\{0\}$ such that every element of $U$
is above some element of $S$.
\item The $\pi$-character $\pi\character{p,X}$ 
of a point of a space or an ultrafilter of a boolean algebra
is the least $\ka\geq\om$ for which it has a local $\pi$-base of size $\leq\ka$.
\item The $\pi$-character of a space (boolean algebra) is
the supremum of the $\pi$-characters of its points (ultrafilters).
\item A caliber of space $X$ is a regular cardinal $\nu$ for
which every sequence of length $\nu$ of nonempty open subsets of $X$
has a subsequence of length $\nu$ such that the open sets from
the subsequence have a common point.
\item A precaliber of a boolean algebra is a regular cardinal $\nu$
for which every sequence of length $\nu$ of nonzero elements of $X$
has a subsequence of length $\nu$ such that the range of
the subsequence can be extended to a filter.
\end{itemize}
\end{definition}

\scepin~\cite{sc}\ proved that if a compactum $Y$ is a continuous image
of an openly generated compacta $X$, then
the weight of $Y$ equals its $\pi$-character 
and every regular uncountable cardinal is a caliber of $Y$.  
The Stone dual of these results is that
every subalgebra of an openly generated boolean algebra
has weight equal to its $\pi$-character and 
has all regular uncountable cardinals as precalibers.

Let us generalize \scepin's results to \laog\ 
compacta.

\begin{theorem}\label{loog}
Let $\lm$ be a regular infinite cardinal and let $Y$ be a
$(\lm,\alo)$\nbd-adic compactum.  The weight of $Y$ then
equals its $\pi$-character or is less than $\lm$, 
and $Y$ has every uncountable regular cardinal 
$\geq\lm$ as a caliber.
\end{theorem}
\begin{remark}
The Stone dual of the theorem is that 
$\weight{A}=\pi\character{A}$ or $\weight{A}<\lm$, 
and all regular uncountable cardinals $\geq\lm$ are precalibers of $A$,
provided $\lm$ is regular and $A$ is a subalgebra of 
some $B$ with the $(\lm,\alo)$-FN.
\end{remark}
\begin{proof}

We modestly generalize \scepin's proofs of the 
theorem's claims for the case $\lm=\al_1$. 
(For calibers, see Theorem 6 of~\cite{sc}; 
for $\pi$-character, see Sections 7 and 8 of the same.)
\scepin's proofs used inverse limits and 
so-called lattices of quotient maps,
not elementary quotients.
Our proof is primarily a work of loose translation
that made generalization to $\lm>\al_1$ easy. 
The two key ideas behind this translation are that
elementary chains induce very nice inverse limits,
and that the set of all elementary quotient maps 
$\pi^X_M$ of a compactum $X$ is a very nice lattice.

Let $f\colon X\rightarrow Y$ be a continuous surjection
from a \laog\ compactum.
First, we prove that every regular $\nu\geq\lm$ is a caliber. 
Since continuous images preserve calibers,
it suffices to prove that $\nu$ is a caliber of $X$.
Proceed by induction of the weight of $X$, 
with the base cases $\weight{X}<\nu$ being trivial.  
Let $(M_i:i<\weight{X})$ be a continuous elementary chain
such that $\card{M_i}<\weight{X}$, $M_i\in M_{i+1}$,
$M_i\elemsub(\bigh,\in,\mcT_X)$, and 
$\card{M_i}\cap\lm\subseteq M_i\cap\lm\in\lm+1$.
This makes $X$ an inverse limit of $(X/M_i:i<\weight{X})$,
with it not being hard to check using elementarity that
the bonding maps $\pi^i_j(\pi^X_{M_i}(x))=(\pi^X_{M_j}(x))$
are well-defined, continuous, and open for $j<i$,
as $M_j\in M_i$.  \scepin\ proved in~\cite{sc} that
a continuous linear inverse limit of compacta
with open bonding maps preserves calibers
(Section 5, Proposition 1), so $\nu$
is a caliber of $X$.

Next, we prove that $\pi\character{Y}=\weight{Y}$.
Assume that $\weight{Y}\geq\lm$ and $\tau$
is a regular uncountable cardinal $\geq\lm$.
We will show that the set $P$ of points $y\in Y$
with $\pi$-character less than $\tau$ 
has weight less than $\tau$.  
This will more than suffice to complete the proof.

Let $(N_i:i<\tau)$ be a continuous elementary chain
such that $\card{N_i}<\tau$, $N_i\in N_{i+1}$,
$N_i\elemsub(\bigh,\in,\mcT_X,\mcT_Y,f,\tau,\sqsubseteq)$
where $\sqsubseteq$ well-orders $\bigh$, and 
$\card{N_i}\cap\lm\subseteq N_i\cap\lm\in\lm+1$.
Let $N=\bigcup_{i<\tau}N_i$.
Choose $M\elemsub(\bigh,\in,\mcT_X,\mcT_Y,f,\vec{N},\sqsubseteq)$
such that $\card{M}\subseteq\tau\cap M\in\tau$ and 
$\card{M}\cap\lm\subseteq M\cap\lm\in\lm+1$.
For each $y\in P$, let $(U_i:i<\pi\character{y,Y})$
be the $\sqsubseteq$-least enumeration of length $\pi\character{y,Y}$
of a local $\pi$-base at $y$ consisting of functionally open sets.
Let $F_y=\bigcap_V\closure{\bigcup_{U_i\subseteq V}\inv{f}U_i}$
where $V$ ranges over all neighborhoods of $y$.
It follows that $F_y$ is a nonempty closed subset of $\inv{f}\{y\}$.

Moreover, for each $y\in P\cap M$, we have $\vec{U}\subseteq M$, 
so $\inv{\pi_M}[\pi_M[\inv{f}[U_i]]]=\inv{f}[U_i]$ 
for all $i$, where $\pi_M=\pi^X_M$.  
Hence, $\inv{\pi_M}[\pi_M[\bigcup_{U_i\subseteq V}\inv{f}[U_i]]]=\bigcup_{U_i\subseteq V}\inv{f}[U_i]$
for all $V$.
Since $\pi_M$ is open,
$\inv{\pi_M}[\pi_M[\closure{\bigcup_{U_i\subseteq V}\inv{f}U_i}]]=\closure{\bigcup_{U_i\subseteq V}\inv{f}U_i}$
for all $V$.
Hence, $\inv{\pi_M}[\pi_M[F_y]]=F_y$.
Since $F_y$ is compact,
$F_y$ has a neighborhood base 
consisting of sets of the form $\inv{\pi_M}[W]$
where $W$ is open.
Since $\pi_M$ is open, we can choose each $W$ to be of
the form $\pi_M[Z]$ where $Z$ is open and $Z\in M$.
Hence, $F_y$ has a neighborhood base that is a subset of $M$.
Hence, $F_y$ has a neighborhood base of size less than $\tau$.
By elementarity, $F_y$ has a neighborhood base of size less than $\tau$
for all $y\in P$.

The space 
$X/N$ is an inverse limit of $(X/N_i:i<\tau)$
with open bonding maps $\pi^i_j(\pi^X_{N_i}(x))=\pi^X_{N_j}(x)$.
\scepin\ proved in~\cite{sc} that
a continuous inverse limit $L$ of length $\tau$
of compacta with open bonding maps is such that
if $\mcA$ is any collection of subsets of $L$
each with neighborhood bases of size less than
$\tau$, then there exists $\mcB\in[\mcA]^{<\tau}$
such that $\closure{\bigcup\mcB}=\closure{\bigcup\mcA}$
(Section 5, Lemma 7).
Therefore, there exists $Q\in[P\cap N]^{<\tau}$ such that 
$\closure{\bigcup_{y\in Q} \pi_N[F_y]}=\closure{\bigcup_{y\in P\cap N} \pi_N[F_y]}$
where $\pi_N=\pi^X_N$.
Since $\tau$ is regular, we can choose $Q$ to be $P\cap N_i$ for some $i<\tau$.
Choose $Q=P\cap N_i$ for the least possible $i$.
Now $Q\in N$, so, setting $F=\closure{\bigcup_{y\in P} F_y}$,
we have $F\in N$ also.
If $W\in N$ is functionally open and intersects $F$,
then $W$ intersects $F_y$ for some $y\in P$; by elementarity, 
we may choose $y\in N$, so $W$ intersects $\bigcup_{y\in P\cap N}F_y$,
which implies that $W$ intersects $\bigcup_{y\in Q}F_y\in N$
because $\inv{\pi_N}[\pi_N[W]]=W$ and,
reusing an argument from the previous paragraph,
$\inv{\pi_N}[\pi_N[F_y]]=F_y$ for all $y\in Q$.
Invoking elementarity once more, \emph{every} functionally open $W$ 
that intersects $F$ intersects $\bigcup_{y\in Q}F_y$, so $F=\closure{\bigcup_{y\in Q}F_y}$.

Since $Q$ is definable from $\vec{N}$, we have $Q\in M$.
Since $\card{Q}<\tau$, we also have $Q\subseteq P\cap M$; 
hence, $\inv{\pi_M}[\pi_M[F_y]]=F_y$ for all $y\in Q$,
which in turn implies that
$\inv{\pi_M}[\pi_M[\closure{\bigcup_{y\in Q} F_y}]]
=\closure{\bigcup_{y\in Q} F_y}$, 
thanks to the openness of $\pi_M$.
Therefore, $\inv{\pi_M}[\pi_M[F]]=F$.
Let us show that $g(\pi_M(x))=f(x)$ defines a 
continuous surjection from $\pi_M[F]$ to $\closure{P}$.  
This will complete the proof because it implies that
$$\weight{P}\leq\weight{\closure{P}}
\leq\weight{\pi_M[F]}\leq\weight{X/M}\leq\card{M}<\tau,$$ 
for $P$ would be a subspace of the continuous image $\closure{P}$ 
of the compact subspace $\pi_M[F]$ of $X/M$.

First, observe that $f$ maps $\bigcup_{y\in P}F_y$ onto $P$,
so $f$ maps $F$ onto $\closure{P}$,
so $g$ maps $\pi_M[F]$ onto $\closure{P}$,
assuming $g$ is well-defined.

Second, we show that $g$ is well-defined.
Assume that $a,b\in F$ and $f(a)\not=f(b)$.
We need to show that $\pi_M(a)\not=\pi_M(b)$.
Choose disjoint regular closed $Y$-neighborhoods 
$A$ and $B$ of $f(a)$ and $f(b)$, respectively.
Because $\bigcup_{y\in Q}F_y$ is dense in $F$, 
$\inv{f}[A]=\closure{\bigcup_{y\in A\cap Q}F_y}$ and 
$\inv{f}[B]=\closure{\bigcup_{y\in B\cap Q}F_y}$.
Since $\pi_M$ is open and $\inv{\pi_M}[\pi_M[F_y]]=F_y$ 
for all $y\in Q$, we have 
$\inv{\pi_M}[\pi_M[\inv{f}[A]]]=\inv{f}[A]$ and 
$\inv{\pi_M}[\pi_M[\inv{f}[B]]]=\inv{f}[B]$.
Therefore, $\pi_M[\inv{f}[A]]$ and $\pi_M[\inv{f}[B]]$ 
are disjoint. Hence, $\pi_M(a)\not=\pi_M(b)$ as desired.

Finally, we show that $g$ is continuous.
Given any open subset $V$ of $X$,
$$\inv{g}[V\cap\closure{P}]=\pi_M[\inv{f}[V\cap \closure{P}]]
=\pi_M[\inv{f}[V\cap f[F]]]=\pi_M[\inv{f}[V]\cap F].$$
Since $\inv{\pi_M}[\pi_M[F]]=F$, we have $\pi_M[\inv{f}[V]\cap F]=\pi_M[\inv{f}[V]]\cap \pi_M[F]$.
Since $f$ is continuous and $\pi_M$ is open, $\pi_M[\inv{f}[V]]$ is open in $X/M$.
Therefore, $\inv{g}[V\cap\closure{P}]$ equals $\pi_M[F]$ intersected with an open subset of $X/M$.
Thus, $g$ is continuous as desired.
\end{proof}

In contrast, the double-arrow space has weight $2^{\alo}$, 
yet all its points have countable $\pi$-character.
This is despite the fact that it is $(\al_2,\al_1)$-openly generated 
because its clopen algebra,
being isomorphic to the real interval algebra, has the WFN.
Therefore, at least in models of $\neg CH$, 
Theorem~\ref{loog} does not generalize 
from $(\lm,\alo)$\nbd-adic compacta to
$(\lambda,\al_1)$-adic compacta. (We naturally
wonder whether there is a ZFC example of this phenomenon.)

\begin{question}
We conclude by drawing attention to the very difficult problem
of ``large'' homogeneous compacta. (See \cite{mill} for a survey.) 
All known homogeneous compacta have $\left(2^{\alo}\right)^+$ as a 
caliber, but over forty years has not answered the question of 
whether this barrier is a theorem (or even a consistency result). 
Indeed, all known homogeneous compacta are continuous images 
of products where each factor is a compactum 
with weight at most $2^{\alo}$, so they are all
$\left(\left(2^{\alo}\right)^+,\alo\right)$-adic,
and, as far the author knows, might all be
$\left(\left(2^{\alo}\right)^+,\alo\right)$-openly generated.
Can a homogeneous compactum fail to be
$\left(\left(2^{\alo}\right)^+,\alo\right)$-openly generated?
Can a homogeneous compactum fail to be
$\left(\left(2^{\alo}\right)^+,\alo\right)$-adic?
\end{question}

\end{document}